\documentclass[a4paper,12pt,fleqn]{article}
\usepackage[latin1]{inputenc}  \usepackage[T1]{fontenc}
\usepackage{lscape}
\usepackage{tabularx}     
\usepackage{pstricks}    
\usepackage{rotating}
\usepackage{amsfonts,amsmath,amsthm,amssymb,dsfont}
\usepackage{marginnote}
\usepackage{rotate}
\usepackage{pgfplots}
\usepackage{tikz}
\usepackage{tikz-cd}
\usepackage[arrow,matrix,curve]{xy} 	
\title{Interval Exchange Transformations and Low-Discrepancy}
\author{Christian Weiß}
\date{\today}

\newtheorem{thm}{Theorem}[section]
\newtheorem{defi}[thm]{Definition}

\newtheorem{prop}[thm]{Proposition}
\newtheorem{cor}[thm]{Corollary}

\newtheorem{exa}[thm]{Example}

\newcommand{\Id}{{\rm{Id}}}

\newcommand{\RR}{{\mathbb{R}}}
\newcommand{\ZZ}{{\mathbb{Z}}}
\newcommand{\NN}{{\mathbb{N}}}
\newcommand{\QQ}{{\mathbb{Q}}}

\makeindex

\begin{document} 

\maketitle

\begin{abstract}
	In \cite{Mas82} and \cite{Vee78} it was proved independently that almost every interval exchange transformation is uniquely ergodic. The Birkhoff ergodic theorem implies that these maps mainly have uniformly distributed orbits. This raises the question under which conditions the orbits yield low-discrepancy sequences. The case of $n=2$ intervals corresponds to circle rotation, where conditions for low-discrepancy are well-known. In this paper, we give corresponding conditions in the case $n=3$. Furthermore, we construct infinitely many interval exchange transformations with low-discrepancy orbits for $n \geq 4$. We also show that these examples do not coincide with $LS$-sequences if $S \geq 2$.
\end{abstract}


\section{Introduction} Low-discrepancy sequences are an essential tool for high dimensional numerical integration. Three families of low-discrepancy sequences are classically used for that purpose, namely Kronecker sequences, digital sequences and Halton sequences (compare \cite{Lar14}, see also \cite{Nie92}). From an ergodic point of view Kronecker sequences are of particular interest: they are realized as orbits of circle rotations which are in turn the simplest examples of interval exchange transformations (two intervals). This raises the question whether further examples of low-discrepancy sequences can be constructed by interval exchange transformations with $n > 2$ intervals.\\[12pt]
In this paper, we at first draw a complete picture for interval exchange transformations with three intervals and give criteria when their orbits yield low-discrepancy sequences. For most combinatorial data also the case $n=3$  corresponds to usual circle rotations but there exists a specific choice of combinatorial data which is different. We draw our attention solely to the latter class of examples. These can be regarded as the first return maps of a circle rotations on some interval $[0,\lambda)$ with $\lambda > 0$ to $[0,1)$. Therefore, our approach consists in considering low-discrepancy sequences in arbitrary intervals and their restriction to $[0,1)$.\\[12pt]
Inspired by the algebraic relations used to define $LS$-sequences (see \cite{Car12}), we then present a canonical construction of further low-discrepancy sequences as orbits of interval exchange transformations with at least four intervals. These examples are similar to Kronecker sequences and give further insight to the interaction of low-discrepancy and ergodicity which has been recently discussed e.g. in \cite{GHL12}, \cite{CIV14} and \cite{Wei17}.

\section{Discrepancy Theory}

\paragraph{(Star-)Discrepancy in arbitrary intervals.} Usually, the concept of discrepancy is applied to sequences $S=(\omega_n)_{n \geq 0}$ in $[0,1)^d$. It may, however, be generalized to arbitrary intervals $I = [\alpha_1,\beta_1) \times \ldots [\alpha_d,\beta_d)$ with $- \infty < \alpha_i < \beta_i < \infty$ for all $i = 1, \ldots, d$ in the following way: let $S = (\omega_n)_{n \geq 0}$ be a sequence in $I$. The \textbf{star-discrepancy} of its first $N$ points is given by
$$D_{N,I}^{*}(S) = \sup_{B^* \subset I} \left| \frac{A_N(B^*,S)}{N} - \frac{\lambda_d(B^*)}{\lambda_d(I)} \right|,$$
where the supremum is taken over all intervals of the form $B^*=[\alpha_1,b_1)\times \ldots \times [\alpha_d,b_d)$ with $\alpha_i < b_i < \beta_i$ for all $i$ and where $\lambda_d$ denotes the $d$-dimensional Lebesgue-measure and $A_N(B,S) := \# \left\{ n \ \mid \ 0 \leq n \leq N, \omega_n \in B \right\}$. Note that the definition is consistent with the usual definition of star-discrepancy for $I=[0,1)^d$ and that we still have $0 \leq D_{N,I}^{*}(S) \leq 1$. Whenever we consider the $d$--dimensional unit-interval we leave away $I$ in the definition of star-discrepancy, $D_N^*(S)$. The following proposition asserts that the star-discrepancy of a sequence is invariant under scaling.
\begin{prop} \label{prop:scaling} Let $S=(\omega_n)_{n \geq 0}$ be a sequence in $I$ with discrepancy $D_{N,I}^*(S)$. Then the component-wise scaled sequence $$\widetilde{S} := \left(-\alpha_1 + \frac{1}{\beta_1 - \alpha_1} z_{1,n}, \ldots, -\alpha_d + \frac{1}{\beta_d - \alpha_d} z_{d,n} \right)_{n \geq 0}$$
has discrepancy $D_N^*(\widetilde{S}) = D_{N,I}^*(S)$ in $[0,1)^d$.
\end{prop}
\begin{proof} This can be proved by a direct calculation. For an interval $B^* = [0,b_1) \times \ldots \times [0,b_d) \subset [0,1)^d$ we set $C^*=[\alpha_1,\alpha_1 + b_1(\beta_1 - \alpha_1)) \times \ldots \times [\alpha_d, \alpha_d + b_d(\beta_d - \alpha_d)) )$. Thus we have
\begin{align*}
	\left| \frac{A_N(B^*,\widetilde{S})}{N} - \lambda(B^*) \right| & = \left| \frac{A_N(B^*,\widetilde{S})}{N} - \frac{\lambda(C^*)}{(\beta_1-\alpha_1) \cdot \ldots \cdot (\beta_d - \alpha_d)} \right|\\
	& = \left| \frac{A_N(C^*,S)}{N} - \frac{\lambda(C^*)}{\lambda(I)} \right|.
\end{align*}	
Taking the supremum over all intervals $B^*$ implies the claim.
 \end{proof} 
If we restrict to the case $d=1$ and if $D_N^*(S)$ satisfies 
$$D_{N}^*(S) = O(N^{-1}\log N)$$
then $S \subset [0,1)$ is called a \textbf{low-discrepancy sequence}. In dimension one this is the best possible rate of convergence as it was proved in \cite{Sch72}, that there exists a constant $c$ with
	$$D_N^*(S) \geq c N^{-1} \log N.$$
The constant $c$ fulfills $0.06 < c < 0.223$ but its precise value is still unknown (see e.g. \cite{Lar14}). 
A theorem of Weyl states that a sequence of points $S=(\omega_n)_{n \geq 0}$ is uniformly distributed if and only if
$$\lim_{N \to \infty} D_{N}^*(S) = 0.$$
Thus, the only candidates for low-discrepancy sequences are uniformly distributed sequences. For a discussion of the situation in higher dimensions see e.g. \cite{Nie92}, Chapter~3. 
 \begin{cor} \label{cor:scaling} Let $S=(\omega_n)_{n \geq 0}$ be a sequence in $I=[\alpha_1,\beta_1)$. Then we have $D_{N,I}^*(S) \geq c N^{-1} \log N$ with $0.06 < c < 0.223$.
 \end{cor}
 Thus it makes sense to speak of low-discrepancy sequences in an arbitrary interval $I=[\alpha_1,\beta_1) \times \ldots \times [\alpha_d,\beta_d)$. Accordingly, the \textbf{discrepancy} of the first $N$ points of a sequence $S=(\omega_n)_{n \geq 0}$ in $I$ is defined by
 $$D_{N,I}(S) := \sup_{B \subset [0,1)} \left| \frac{A_N(B,S)}{N} - \frac{\lambda_d(B)}{\lambda_d(I)} \right|,$$
 where the supremum is taken over all subintervals $B = [a_1,b_1) \times \ldots \times [a_d,b_d)$ with $\alpha_i \leq a_i < b_i  < \beta_i$ for all $i$. It is straightforward to see that $D_{N,I}^*(S) \leq D_{N,I}(S) \leq 2^d D_{N,I}^*(S)$.
 \paragraph{Low-discrepancy in subintervals.} The low-discrepancy property of a sequence is often regarded as \textit{being as uniformly distributed as possible}. We aim to make this statement more precise here. Let $(\omega_n)_{n \geq 0}$ be a low-discrepancy sequence in the interval $I$ and let $I_1 \subset I$ be a half-open subinterval with its left endpoint included. Since $(\omega_n)_{n \geq 0}$ is uniformly distributed there are infinitely many $\omega_i$ in $I_1$. Hence the numbers
 $$i_0 := \min \left\{ j \in \NN_0 | \omega_j \in I_1 \right\}, \quad i_j := \min \left\{ j \in \NN_0 | j > i_{j-1}, \omega_j \in I_1 \right\} \ (j \geq 1)$$
 form an infinite sequence $(i_j)_{j \geq 0}$ in $\NN_0$. As a consequence, the subsequence $(\omega^1_j)_{j \geq 0} := (\omega_{i_j})_{j \geq 0}$ is well-defined and lies in $I_1$. It is called the \textbf{restriction of $(\omega_n)_{n \geq 0}$ to $I_1$}. By this construction low-discrepancy is preserved.
 \begin{prop} \label{prop:subsequence} Let $I_1 \subset I$ be an arbitrary half-open subinterval and let $(\omega_n)_{n \geq 0}$ be a low-discrepancy sequence in $I$. Then the restriction $(\omega^1_n)_{n \geq 0}$ is a low-discrepancy sequence in $I_1$. 
 \end{prop}
 \begin{proof}  In order to facilitate notation we leave away the indexes of sequences in this proof. 
  The low-discrepancy of $\omega$ in $I$ implies
  \begin{align} \label{eq1}
  \left| \frac{A_N(\omega,I_1)}{N} - \lambda(I_1) \right| \leq c \frac{\log N}{N}
  \end{align}
  for all $N \in \NN$ with fixed $c \in \RR$ independent of $N$. Now let $N_1 \in \NN$ be arbitrary and choose $N \geq N_1$ such that $N_1 = {A_N(\omega,I_1)}$. We need to bound the following quantity
 \begin{align*}
 \widetilde{D}_{N_1,I^*}^* & := \sup_{I^* \subset I_1} \left| \frac{A_{N_1}(\omega^1,I^*)}{N_1} - \frac{\lambda(I^*)}{\lambda(I_1)}\right|\\
 & \leq \sup_{I^* \subset I_1} \left| \frac{A_{N_1}(\omega^1,I^*)}{N_1} - \frac{N\lambda(I^*)}{N_1}\right| + \sup_{I^* \subset I_1} \left| \frac{N\lambda(I^*)}{N_1} - \frac{\lambda(I^*)}{\lambda(I_1)}\right|.
 \end{align*} 
 For every subinterval $I^* \subset I_1$ we have
 \begin{align*}
 A_{N_1}(\omega^1,I^*) = A_{N}(\omega,I^*)
 \end{align*}
 yielding
 \begin{align*}
 \widetilde{D}_{N_1,I^*}^* & \leq \frac{N}{N_1} \sup_{I^* \subset I_1} \left| \frac{A_{N}(\omega,I^*)}{N} - \lambda(I^*)\right| + \frac{N}{N_1} \sup_{I^* \subset I_1} \frac{\lambda(I^*)}{\lambda(I_1)} \left| \lambda(I_1) - \frac{N_1}{N}\right|\\
 & =  \frac{N}{N_1} \sup_{I^* \subset I_1} \left| \frac{A_{N}(\omega,I^*)}{N} - \lambda(I^*)\right| + \frac{N}{N_1} \sup_{I^* \subset I_1} \frac{\lambda(I^*)}{\lambda(I_1)} \left| \frac{A_N(\omega,I_1)}{N} - \lambda(I_1) \right|\\
 & \stackrel{\eqref{eq1}}{\leq} \frac{N}{N_1} c \frac{\log N}{N} + \frac{N}{N_1} \cdot 1 \cdot c \frac{\log N}{N}.
 \end{align*}
 Therefore, it suffices to bound $N/N_1$. We show here that for sufficiently large $N$ the expression $N_1/N$ is arbitrarily close to a positive constant. Indeed, this claim follows from 
 \begin{align*}
	 \left| \frac{N_1}{N} - \lambda(I_1) \right| = \left| \frac{A_N(\omega,I_1)}{N} - \lambda(I_1) \right| \leq c \frac{\log N}{N}.
 \end{align*}
 since $\lambda(I_1)$ is some positive constant that does not depend on $N$ and $\log N / N$ converges to zero. This completes the proof. 
 \end{proof}
 
  \paragraph{Kronecker sequences.} A classical example of low-discrepancy sequences are \textbf{Kronecker sequences}: given $z \in \RR$, let $\left\{ z \right\} := z - \lfloor z \rfloor$ denote the fractional part of $z$. A Kronecker sequence is a sequence of the form $(z_n)_{n \geq 0} = (\left\{ n z \right\})_{n \geq 0}$. If $z \notin \QQ$ and $z$ has bounded partial quotients in its continued fraction expansion, the sequence $(z_n)$ has low-discrepancy in $[0,1)$. More precisely, the following statement holds:
    \begin{thm} \label{thm:kronecker} (\cite{DT97}, Corollary 1.65) Let $z \in \RR$ be irrational with continued fraction expansion $z =[a_0;a_1;a_2;\ldots]$. Then the sequence $(\left\{ n z \right\})_{n \geq 0}$ is low-discrepancy if and only if the moving average
    	$$a_m^{(1)} = \frac{1}{m} \sum_{j=1}^m a_j$$
    is a bounded sequence.  	
  \end{thm}  
  The sequence $(z_n)_{n \geq 0}$ corresponds to the rotation of the unit circle by the angle $2 \pi z$ if we identify $[0,1)$ with $\RR / \ZZ$. Furthermore, it is the simplest example of an interval exchange transformation (see Chapter~\ref{cha:iet}).
  
  \paragraph{LS-sequences.}  Another way to construct uniformly distributed sequences goes back to the work of Kakutani \cite{Kak76} and was later on generalized in \cite{Vol11} in the following sense. 
  \begin{defi}
  	Let $\rho$ denote a non-trivial partition of $[0,1)$. Then the \textbf{$\rho$-refinement} of a partition $\pi$ of $[0,1)$, denoted by $\rho \pi$, is defined by subdividing all intervals of maximal length positively homothetically to $\rho$. 
  \end{defi}
  The resulting sequence of partitions is denoted by $\left\{ \rho^n \pi \right\}_{n \in \NN}$. We now turn to a specific class of examples of $\rho$-refinement which was introduced in \cite{Car12}. 
  \begin{defi} 
  	Let $L \in \NN, S \in \NN_0$ and $\beta$ be the solution of $L \beta + S \beta^2 = 1$. An \textbf{$LS$-sequence of partitions} $\left\{ \rho_{L,S}^n \pi \right\}_{n \in \NN}$ is the successive $\rho$-refinement of the trivial partition $\pi = \left\{ [0,1) \right\}$ where $\rho_{L,S}$ consists of $L+S$ intervals such that the first $L$ intervals have length $\beta$ and the successive $S$ intervals have length $\beta^2$.
  \end{defi}
  The partition $\left\{ \rho_{L,S}^n \pi \right\}$ consists of intervals only of length $\beta^n$ and $\beta^{n+1}$. Its total number of intervals is denoted by $t_n$, the number of intervals of length $\beta^n$ by $l_n$ and the number of intervals of length $\beta^{n+1}$ by $s_n$. A specific ordering of the endpoints of the partition yields the $LS$-sequence of points. 
  \begin{defi} \label{def:ordering_LS} Given an $LS$-sequence of partitions $\left\{ \rho_{L,S}^n \pi \right\}_{n \in \NN}$, the corresponding \textbf{$LS$-sequence of points} $(\xi^n)_{n \in \NN}$ is defined as follows: let $\Lambda_{L,S}^1$ be the first $t_1$ left endpoints of the partition $\rho_{L,S} \pi$ ordered by magnitude. Given $\Lambda_{L,S}^n = \left\{ \xi_1^{(n)}, \ldots, \xi_{t_n}^{(n)} \right\}$ an ordering of $\Lambda_{L,S}^{n+1}$ is then inductively defined as
  	\begin{align*}
  	\Lambda_{L,S}^{n+1} = \left\{ \right.&  \xi_1^{(n)}, \ldots, \xi_{t_n}^{(n)},\\
  	&\left. \psi_{1,0}^{(n+1)} (\xi_1^{(n)}), \ldots, \psi_{1,0}^{(n+1)} (\xi_{l_n}^{(n)}), \ldots, \psi_{L,0}^{(n+1)} (\xi_1^{(n)}), \ldots, \psi_{L,0}^{(n+1)} (\xi_{l_n}^{(n)}), \right. \\
  	&\left. \psi_{L,1}^{(n+1)}	(\xi_1^{(n)}), \ldots, \psi_{L,1}^{(n+1)}	(\xi_{l_n}^{(n)}), \ldots, \psi_{L,S-1}^{(n+1)}	(\xi_1^{(n)}), \ldots, \psi_{L,S-1}^{(n+1)}	(\xi_{l_n}^{(n)}) \right\},
  	\end{align*}
  	where
  	$$\psi^{(n)}_{i,j}(x) = x + i\beta^n + j\beta^{n+1}, \qquad x \in \RR.$$
  \end{defi}
  As the definition of $LS$-sequences might not be completely intuitive at first sight, we illustrate it by an explicit example.
  \begin{exa} \label{exa:KF_sequence} For $L=S=1$ the $LS$-sequence is the so-called \textbf{Kakutani-Fibonacci sequence} (see \cite{CIV14}). We have
  	\begin{align*}
  	\Lambda^1_{1,1} & = \left\{ 0, \beta \right\}\\
  	\Lambda^2_{1,1} & = \left\{ 0, \beta, \beta^2 \right\}\\
  	\Lambda^3_{1,1} & = \left\{ 0, \beta, \beta^2, \beta^3, \beta + \beta^3 \right\}\\
  	\Lambda^4_{1,1} & = \left\{ 0, \beta, \beta^2, \beta^3, \beta + \beta^3, \beta^4,\beta+\beta^4,\beta^2+\beta^4  \right\}
  	\end{align*}
  	and so on.
  \end{exa}
  More precisely, $LS$-sequences are not only uniformly distributed but in many cases even low-discrepancy.
  \begin{thm}[Carbone, \cite{Car12}] If $L \geq S$, then the corresponding $LS$-sequence has low-discrepancy.
  \end{thm}
  It has been pointed out that for parameters $S=0$ and $L=b$, the corresponding $LS$-sequence coincides with the classical van der Corput sequences. In \cite{Wei17} it was proven that $LS$-sequences for $S=1$ coincide with symmetrized Kronecker sequences up to permutation  and that neither van der Corput sequences nor Kronecker sequences occur for $S \geq 2$.
  
  \section{Interval Exchange Transformations} \label{cha:iet}
  Let $I \subset \RR$ be an interval of the form $[0,\lambda^*)$ and let $\left\{ I_\alpha | \alpha \in \mathcal{A} \right\}$ be a finite partition of $I$ into sub-intervals indexed by some finite alphabet $\mathcal{A}$. An \textbf{interval exchange transformation} is a map $f: I \to I$ which is a translation on each subinterval $I_\alpha$. It is determined by its combinatorial data and its length data. The \textbf{combinatorial data} consists of two bijections $\pi_0, \pi_1: \mathcal{A} \to \left\{1,\ldots,n \right\}$, where $n$ is the number of elements of $\mathcal{A}$ and the \textbf{length data} are numbers $(\lambda_\alpha)_{\alpha \in A}$ with $\lambda_\alpha > 0$ and $\lambda^* = \sum_{\alpha \in \mathcal{A}} \lambda_\alpha$. The number $\lambda_\alpha$ is the length of the subinterval $I_\alpha$ and the pair $\pi = (\pi_0,\pi_1)$ describes the ordering of the subintervals before and after the map $f$ is iterated. In analytical terms, consider for $\epsilon \in \left\{ 0, 1 \right\}$ the maps $j_\epsilon$ which are given on $I_\alpha$ by
  \begin{align} \label{eq5}
  j_\epsilon(x) = x + \sum_{\pi_\epsilon(\beta) < \pi_\epsilon(\alpha)} \lambda_\beta.
  \end{align}
  The interval exchange transformation corresponding to the combinatorial and length data equals $f = j_1 \circ j_0^{-1}$. It can also be described completely by its \textbf{translation vector} $w$ with components
  $$w_\alpha = \sum_{\pi_1(\beta)<\pi_1(\alpha)} \lambda_\beta -  \sum_{\pi_0(\beta)<\pi_0(\alpha)} \lambda_\beta.$$
   Note that combinatorial data is \textit{not} uniquely determined by $f$ (see e.g. \cite{Via06}, Example~1.3). For $\mathcal{A} = \left\{A, B \right\}$ and $\pi_0(A) = \pi_1(B) = 1$ and $\pi_1(A) = \pi_0(B) = 2$ the interval exchange transformation becomes the rotation of $\RR / \lambda^* \ZZ$ by $\lambda_B$.\\[12pt]
  If the combinatorial data satisfies 
  \begin{align} \label{eq4}
  \pi_0^{-1}(\left\{1, \ldots, k \right\}) = \pi_1^{-1}(\left\{1, \ldots, k \right\})
  \end{align} 
  for some $k < d$, the interval exchange transformation splits into two interval exchange transformations of simpler combinatoric. The analysis of interval exchange transformations is therefore usually restricted to \textbf{admissible} combinatorial data, for which \eqref{eq4} does not hold for any $k < d$.  Further details on interval exchange transformation can be found e.g. in \cite{Mas82}, \cite{Via06} and \cite{Yoc06}.\\[12pt]  
   Recently, it has been discussed in some articles, see e.g. \cite{GHL12}, \cite{CIV14} and \cite{Wei17}, that there is a connection between interval exchange transformations and low-discrepancy. 
   \begin{thm} \textbf{(\cite{CIV14}, Theorem~17)} For $L=S=1$ the $LS$-sequence coincides with an orbit of an ergodic interval exchange transformation $T: [0,1) \to [0,1)$ consisting of infinitely many intervals.
   \end{thm}
   \begin{thm} \textbf{(\cite{Wei17}, Theorem~1.6)} For $S=1$ and arbitrary $L$ the $LS$-sequence coincides with an orbit of an interval exchange transformation $T: [0,1) \to [0,1)$ consisting of two intervals.
   \end{thm}
   
   \paragraph{The case $n=3$.} We will show next that also the second simplest examples of interval exchange transformations, i.e. $n=3$, yield low-discrepancy sequences in $[0,1)$. In that case, the argumentation is based on Proposition~\ref{prop:subsequence}.\\[12pt]
   Assume that $\mathcal{A} = \left\{ A,B,C \right\}$. Without loss of generality we have $\pi_0(A) = 1, \pi_0(B) = 2, \pi_0(C) = 3$. Only $3$ of the $6$ bijections $\pi_1: \mathcal{A} \to \left\{1,2,3\right\}$ yield admissible combinatorial data. Two of these choices for $\pi_1$ yield again rotations on the circle $\RR/\lambda^* \ZZ$ (see e.g. \cite{Yoc06}) and therefore Kronecker sequences. The case we are interested in is thus (compare Figure 1)
   $$\pi_1(A) = 3,\pi_1(B)= 2, \pi_1(C) = 1.$$   
   According to \cite{Yoc06}, Section~2, the corresponding interval transformation $f: I \to I$ can be described as follows: let $\hat{I} :=[0,\lambda^* + \lambda_B]$ and $T: \hat{I} \to \hat{I}$ be defined as the rotation of $\RR / (\lambda^* + \lambda_B) \ZZ$ by $\lambda_B + \lambda_C$. Then for $y \in [0,\lambda_A)$ or $y \in [\lambda_A + \lambda_B, \lambda^*)$ we have $f(y) = T(y)$. For $y \in [\lambda_A, \lambda_A + \lambda_B)$ the identity $f(y)= T^2(y)$ holds. Therefore, $f$ is the first return map of $T$ in $I$. In other words, the sequence $(f^n(y))_{n \geq 0}$  is the restriction of $(T^n(y))_{n \geq 0}$ to $I$. From Proposition~\ref{prop:subsequence} and Theorem~\ref{thm:kronecker} we thus deduce.
    \begin{center}
   	\begin{tikzpicture}[scale=1.0]
   	\draw[thick] (0,0)--(8,0);     	
   	\draw[thick] (0,0)--(0,0.25);     	
   	\draw (0.75,0)--(0.75,0) node[above] {$A$};
   	\draw[thick] (1.5,0)--(1.5,0.25);
   	\draw (2.875,0)--(2.875,0) node[above] {$B$};
   	\draw[thick] (4.25,0)--(4.25,0.25);  	
   	\draw (6.125,0)--(6.125,0) node[above] {$C$};    	
   	\draw[thick] (8,0)--(8,0.25);
   	
   	\draw[thick] (0,-0.25)--(8,-0.25);    	
   	\draw[thick] (0,-0.25)--(0,-0.5);     	
   	\draw (1.875,-0.25)--(1.875,-0.25) node[below] {$C$};    	
   	\draw[thick] (3.75,-0.25)--(3.75,-0.5);  	
   	\draw (5.125,-0.25)--(5.125,-0.25) node[below] {$B$}; 
   	\draw[thick] (6.5,-0.25)--(6.5,-0.5);
   	\draw (7.25,-0.25)--(7.25,-0.25) node[below] {$A$}; 
   	\draw[thick] (8,-0.25)--(8,-0.5); 
   	
   	\end{tikzpicture}\\    	
   	Figure 1. Interval exchange transformation for $n=3$ \label{fig:exa1}
   \end{center}
   
   \begin{thm} \label{thm:n=3} Given $\mathcal{A} = \left\{ A,B,C \right\}$, the combinatorial data
   \begin{align*}
	   \pi_0(A) = 1, \pi_0(B) = 2, \pi_0(C) = 3, \quad \pi_1(A) = 3,\pi_1(B)= 2, \pi_1(C) = 1,
   \end{align*}
   and the length data $\lambda_A, \lambda_B, \lambda_C$, the corresponding interval exchange transformation $f: [0,\lambda^*) \to [0,\lambda^*)$ yields a low-discrepancy sequence $(f^n(y))_{n \geq 0}$ for all $y \in [0, \lambda^*)$ if and only if $\frac{\lambda_B+\lambda_C}{\lambda^* + \lambda_B}$ is irrational and has bounded moving average of its continued fraction expansion.  
   \end{thm}
   In particular, if we choose $\lambda_A,\lambda_B,\lambda_C$ such that $\lambda^* = 1$, then $(f^n(y))_{n \geq 0}$ is a (classical) low-discrepancy sequence in $[0,1)$ which comes from the restriction of a low-discrepancy sequence in $[0,1+\lambda_B)$. By \cite{Nie92}, Corollary~3.5, the rate of convergence of star-discrepancy is expected to be the better the smaller the coefficients in the continued fraction expansion of  $\tfrac{\lambda_B+\lambda_C}{1 + \lambda_B}$ are. Hence, for a fair comparison with the case $n=2$, i.e. the circle rotation by $\gamma$, we should impose 
   $$\frac{\lambda_B+\lambda_C}{1 + \lambda_B} = \gamma.$$
   The additional condition lowers the degree of freedom to $1$. We may e.g. choose $\lambda_C$ with $\lambda_C < \min(\tfrac{\gamma}{1-\gamma},\tfrac{1-\gamma}{\gamma})$ arbitrarily and by that automatically fix $\lambda_A$ and $\lambda_B$. In Figure~2 we plot the star-discrepancy of the orbit arising from the circle rotation by the golden mean $\gamma = \tfrac{\sqrt{5}-1}{2}$ in comparison to the case $n=3$ with $\lambda_C = \gamma / 2$ and $\lambda_C = \gamma / 4$ respectively. Taking into account the graph of $\log(N)/N$ we clearly see that all sequences are indeed low-discrepancy. Although the star-discrepancies have similar asymptotic behavior, there seems to be more variation for the case $n=3$. Yet, the degree of freedom for the choice of $\lambda_C$ allows for optimization of star-discrepancy.
   \begin{center}
   	 \includegraphics[scale = 0.5]{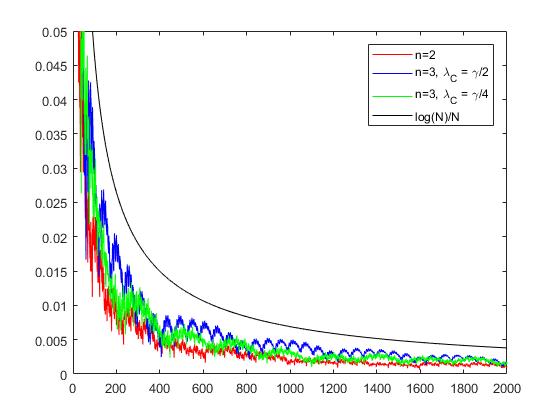}\\
   	 Figure 2. Comparison of star-discrepancies up to $N = 2.000$.
   \end{center}
   
   \paragraph{The general case.} Under some rather mild conditions, interval exchange transformations yield low-discrepancy sequences for $n=2,3$. Still the requirements imposed in Theorem~\ref{thm:n=3} are strictly stronger than the so-called Keane condition (compare \cite{Via06}, Chapter 3). Nevertheless, there is reasonable hope that a similar result can be proven also for $n \geq 4$ due to Birkhoff's ergodic theorem and the following theorem proven independently by Masur \cite{Mas82} and Veech \cite{Vee78}. Let the combinatorial data be fixed and \textit{almost every} refer to the choice of length data with respect to the Lebesgue measure. 
   \begin{thm} \textbf{(Masur, Veech)} Almost every interval exchange transformation is uniquely ergodic.
   \end{thm}
   Generalizing Theorem 1.6 in \cite{Wei17} we may however not expect to get $LS$-sequences with $S \geq 2$ by this construction. 
   \begin{thm}
   		For arbitrary $L$ and $S \geq 2$ the corresponding $LS$-sequence of partitions does not coincide with an orbit of any interval exchange transformation.   
   \end{thm}
	\begin{proof}
		Let $L$ and $S \geq 2$ be fixed. We have
		\begin{align} \label{eq6} 
		\beta^{n+2} = \frac{\beta^{n+1}-L\beta^n}{S}
		\end{align}
		and the denominator grows arbitrarily large with increasing $n$. Suppose that the claim is not true and let $(\lambda_\alpha)_{\alpha \in \mathcal{A}}$ denote the length data of the corresponding interval exchange transformation and $x$ be an arbitrary point in $[0,1)$.  We see from \eqref{eq5} that every point in the orbit of $x$ may be written as a linear combination of $x$ and the $(\lambda_\alpha)_{\alpha \in \mathcal{A}}$ with integral coefficients. This contradicts \eqref{eq6} because by construction the $LS$-sequence contains $\beta^l$ for all $l \in \NN$.  
	\end{proof}
Still the algebraic relation $L \beta + S \beta^2 = 1$, which is a main ingredient for the definition of $LS$-sequences, can be used to get interval exchange transformations potentially yielding low-discrepancy sequences. Recall the fact that every interval exchange transformation can be normalized by choosing $\mathcal{A} = \left\{ 1, \ldots, d \right\}$ and $\pi_0 = \Id$ (see \cite{Via06}). If we set $\lambda_i = \beta$ for $i=1,\ldots, L$ and $\lambda_i = \beta^2$ for $i = L+1, \ldots, L+S$ and let $\pi_1$ be arbitrary then the intervals before applying $f$ corresponds to the intervals of $\rho_{L,S}$. The orbit of every left endpoint of $\rho_{L,S}$ is contained in the set
\begin{align*}
\mathcal{J}_{L,S} :&= \left\{ \left\{m \beta + n \beta^2\right\} \ | \ m,n \in \ZZ \right\} \\
&= \left\{ \left\{m \beta + n \beta^2\right\} \ | \ m,n \in \ZZ, 0 \leq n < S \right\} \\
&= \bigcup_{n=1}^{S-1} \left\{ \left\{m \beta + n \beta^2\right\} \ | \ m \in \ZZ \right\}.
\end{align*}
By Dirichlet's approximation theorem, the set $\mathcal{J}_{L,S}$ is dense in $[0,1]$. It can even be equipped with an ordering that turns it into a low-discrepancy sequence. 
\begin{thm} The set $\mathcal{J}_{L,S}$ can be ordered such that the corresponding sequence is low-discrepancy.
\end{thm}
\begin{proof} It follows from Lagrange's Theorem on continued fractions that $\beta$ has eventually periodic continued fraction expansion (see e.g. \cite{Ste92}). Therefore, for each $n \in \ZZ$ and $N \in \NN$, the set 
	$$ \mathcal{J}_{L,S}^n := \left\{ \left\{m \beta + n \beta^2\right\} \ | \ m \in \ZZ \right\}$$
	has discrepancy $D_N(\mathcal{J}_{L,S}^n) \leq c \log{N} /N$ for some $c \in \RR$ by Theorem~\ref{thm:kronecker}. Now, we define a sequence $(x_i)$ of the left endpoints of $\mathcal{J}_{L,S}$ by
	$$ x_i := \left\{ \lfloor i / S \rfloor \beta + (i - kS) \beta^2  \right\} \quad \textrm{for} \ kS \leq i < (k+1)S, \ k \in \NN_0 $$
	The triangle inequality for discrepancies (see \cite{KN74}, Theorem 2.6) implies that $(x_i)$ is indeed a low-discrepancy sequence.
\end{proof}
Next, we show how to construct interval exchange transformations which have an orbit that coincides with $\mathcal{J}_{L,S}$ as a set. The example for $n=4$ is illuminating.
   \begin{center}
	\begin{tikzpicture}[scale=1.0]
	\draw[thick] (0,0)--(8,0);     	
	\draw[thick] (0,0)--(0,0.25);     	
	\draw (1.464,0)--(1.464,0) node[above] {$1$};
	\draw[thick] (2.928,0)--(2.928,0.25);
	\draw (4.392,0)--(4.396,0) node[above] {$2$};
	\draw[thick] (5.864,0)--(5.864,0.25);  	
	\draw (6.396,0)--(6.396,0) node[above] {$3$};  
	\draw[thick] (6.928,0)--(6.928,0.25);   	
	\draw (7.464,0)--(7.464,0) node[above] {$4$};  
	\draw[thick] (8,0)--(8,0.25);
	
	\draw[thick] (0,-0.25)--(8,-0.25);     	
	\draw[thick] (0,-0.25)--(0,-0.5);     	
	\draw (1.464,-0.25)--(1.464,-0.25) node[below] {$2$};
	\draw[thick] (2.928,-0.25)--(2.928,-0.5);
	\draw (3.465,-0.25)--(3.465,-0.25) node[below] {$4$};
	\draw[thick] (4,-0.25)--(4,-0.5);  	
	\draw (5.464,-0.25)--(5.464,-0.25) node[below] {$1$};  
	\draw[thick] (6.928,-0.25)--(6.928,-0.5);   	
	\draw (7.464,-0.25)--(7.464,-0.25) node[below] {$3$};  
	\draw[thick] (8,-0.25)--(8,-0.5);

	\end{tikzpicture}\\    	
	Figure 3. Interval exchange transformation for $n=4$ \label{fig:exa}
\end{center}
\begin{exa} \label{exa:general_iet} Let $n=4$, $L=S=2$ and $2 \beta + 2\beta^2 = 1$. We choose $\pi_1$ by setting $\pi_1(1)=2, \pi_1(2) = 4, \pi_1(3)= 1$ and $\pi_1(4) = 3$ (see Figure~3). The corresponding interval exchange transformation $f_{2,2}$ is admissible with translation vector
	\begin{align*}
	 & w_1 = \beta + \beta^2, \quad w_2 = -\beta, \quad w_3 = \beta^2 \quad  w_4 = - \beta - \beta^2.
	\end{align*}
	Note that $f_{2,2}(x) \neq x$ for all $x \in [0,1)$. For $x \in I_1$ we have $f_{2,2}(x) \in I_2$ or $f_{2,2}(x) \in I_3$. Furthermore, the following implications hold
    \begin{align*}
	    & x \in I_2 \quad \Rightarrow \quad f_{2,2}(x) \in I_1\\
	    & x \in I_3 \quad \Rightarrow \quad f_{2,2}(x) \in I_4\\
	    & x \in I_4 \quad \Rightarrow \quad f_{2,2}(x) \in I_2.
    \end{align*}
    This means $f_{2,2}$ always maps an element of $I_2 \cup I_3$ to $I_1 \cup I_4$ and vice versa. Hence, the translation vector of $f_{2,2}^2$ is restricted to the following possibilities, where $\equiv$ stands for equivalence modulo $1$:
    \begin{align*}
	    & w_1 + w_2 = \beta + \beta^2 -  \beta = \beta^2\\
	    & w_1 + w_3 = \beta + \beta^2 + \beta^2 \equiv -\beta\\
	    & w_3 + w_4 =  \beta^2 -\beta - \beta^2 = - \beta\\	    
	    & w_4 + w_2 = - \beta - \beta^2 - \beta \equiv \beta^2 
    \end{align*}
    We define the sequence $(x_n)_{n\in\ZZ}$ by $x_0 = \beta$ and $x_k = f_{2,2}^k(x_0)$ for all $k \in \ZZ$ and claim that every pair $(x_{2k},x_{2k+1})$ is either of the form
    $$(\left\{(-k+1)\beta\right\},\left\{(-k+2)\beta+\beta^2\right\}) \quad \textrm{or} \quad (\left\{(-k+2)\beta+\beta^2\right\},\left\{(-k+1)\beta\right\}).$$
    We prove this by induction on $k$. For $k=0$ the first possibility holds, i.e. $(x_{0},x_1) = (\beta,2\beta + \beta^2)$. Now assume that the claim is true for $k$. At first, we consider the case $(x_{2k},x_{2k+1}) = (\left\{(-k+1)\beta\right\},\left\{(-k+2)\beta+\beta^2\right\})$. In view of the translation vector of $f_{2,2}^2$, the entry $x_{2k+2}$ is either equal to 
    \begin{align*}
    & \left\{ (-k+1)\beta - \beta \right\}  = \left\{ (-(k+1)+1)\beta \right\} =: y_{2k+2} \quad \textrm{or}\\
    & \left\{ (-k+1)\beta + \beta^2 \right\} = \left\{ (-(k+1)+2)\beta + \beta^2 \right\} =: z_{k+2}.
    \end{align*}
    Since $x_{2k+3} = f_{2,2}(x_{2k+2}) \neq x_{2k+2}$, the entry $x_{2k+3}$ is $z_{2k+2}$ if $x_{2k+2} = y_{2k+2}$ and vice versa. The argumentation in the second case, namely $(x_{2k},x_{2k+1}) = (\left\{(-k+2)\beta+\beta^2\right\},\left\{(-k+1)\beta\right\})$ is similar. In summary, it follows that $\left\{ f_{2,2}^n(\beta) \ | \ n \in \ZZ \right\}$ equals $\mathcal{J}_{L,S}$ as a set.
 
\end{exa}
Although the combinatorial data in Example~\ref{exa:general_iet} is admissible, there is as few interaction between the $\beta$-part and the $\beta^2$-part of the partition as possible. For arbitrary $n$ we specify $\pi_1$ by
\begin{align*}
	& \pi_1(i) = i + 1, \quad i = 1, \ldots, L-1,\\
	& \pi_1(L) = L+S,\\
	& \pi_1(L+1) = 1,\\
	& \pi_1(i) = i + 1, \quad i = L+2, \ldots L + S.
\end{align*}
Thus, the translation vector of the corresponding interval exchange transformation is
\begin{align*}
& w_1 = (L-1)\beta + \beta^2\\
& w_i = - \beta, \quad i = 2, \ldots, L,\\
& w_i = \beta^2 , \quad i = L+1, \ldots L + S-1,\\
& w_{L+S} = -(S-1)\beta^2-\beta.
\end{align*}
Now we fix an arbitrary $r \in \ZZ$ and choose $q_0 \in \ZZ$ such that $\left\{ -r\beta - q_0 \beta^2 \right\} \in I_1$ but $\left\{ -r\beta - (q_0+1) \beta^2 \right\} \notin I_1$ and set $x_0 = \left\{ -r\beta - q_0 \beta^2 \right\}.$ 
\begin{thm} \label{thm:orbit}
	If $L \geq S$, then the orbit  $\left\{ f_{L,S}^n(x_0) \ | \ n \in \ZZ \right\}$ equals $\mathcal{J}_{L,S}$ as a set.
\end{thm}
\begin{proof}
	In order to facilitate notation we leave away brackets $\left\{ \cdot \right\}$, i.e. all expressions have to be interpreted modulo $1$. The function $f_{L,S}$ maps 
	$$x_0 \mapsto (-r+(L-1))\beta - (q_0-1) \beta^2$$
	with $f_{L,S}(x_0) \in I_L$. This implies
	$$f_{L,S}^i(x_0) = (-r+(L-1)-(i-1)) - (q_0-1) \beta^2, \quad i = 1, \ldots, L.$$
	Note that $m \beta^2 > \beta$ is equivalent to $m \geq L+1$ by the definition of $L$ and $S$. Therefore we again have $f_{L,S}(f_{L,S}^L(x_0)) \in I_L$ and the orbit continues as
	$$f_{L,S}^i(x_0) = (-r+(L-1)-(i-1)) - (q_0-k) \beta^2, \quad i = (k-1)L + 1, \ldots, kL$$
	until $k=L$. When applying $f_{L,S}$ once more the orbit intersects $I_{L+1}$ instead of $I_L$, i.e. 
	$$f_{L,S}^{L^2+1} = (-r+(L-1))\beta - (q_0-L-1)\beta^2 \in I_{L+1}.$$
	For the next $S-1$ iterations the term $\beta^2$ is added such that
	$$f_{L,S}^{L^2+1+i} = (-r+(L-1))\beta - (q_0-L-1+i)\beta^2, \quad i = 1, \ldots, S-1.$$
	Since $f_{L,S}^{L^2+S} \in I_{K+S}$, the orbit jumps back to $I_L$ next. The distance between $f_{L,S}^{L^2+S+1}(x_0)$ and $(L-1)\beta$ is less than $\beta^2$ and we have
	$$f_{L,S}^{L^2+S+i} = (-r + (L-1) - i) \beta - (q_0-L-1)\beta^2 , \quad i = 1, \ldots, L.$$
	Hence, after $L^2+L+S$ applications of $f_{L,S}$ the orbit reaches again $I_1$ and $f_{L,S}^{L^2+S+1}(x_0) < \beta^2$. We are back in the same situation as at the start with $-r\beta-q\beta^2$ replaced by $-(r+1)\beta - (q_0-L-1)$ and it is clear that every $y \in \mathcal{J}_{L,S}$ is contained in the orbit of $f_{L,S}$. 
	 
\end{proof}

\begin{cor}
	If $L \geq S$, then there exists a ordering of $f_{L,S}^i(x_0)$ such that the corresponding sequence $(x_i)_{i \in \ZZ}$ is a low-discrepancy sequence.
\end{cor}




\textsc{Hochschule Ruhr West, Duisburger Str. 100, D-45479 M\"ulheim an der Ruhr}\\
\textit{E-mail address:} \texttt{christian.weiss@hs-ruhrwest.de}

\end{document}